\documentclass{amsart}

\usepackage{amssymb}
\usepackage{eepic}
\usepackage{color}
\usepackage[linktocpage=true]{hyperref}

\allowdisplaybreaks

\numberwithin{equation}{section}

\newtheorem{theorem}{Theorem}[section]
\newtheorem{lemma}[theorem]{Lemma}

\newtheorem{proposition}[theorem]{Proposition}
\newtheorem{corollary}[theorem]{Corollary}
\newtheorem{remark}[theorem]{Remark}

\newcommand{\R}{\mathbb{R}}

\newcommand{\summ}{\sum\nolimits}

\def\1{\mathbf{1}}

\def\M{\mathcal{M}}
\def\A{\mathcal{A}}

\def\supp{\mathrm{supp}}
\def\llangle{\langle\langle}
\def\rrangle{\rangle\rangle}

\def\mean{- \hskip-10.5pt \int}

\addtocounter{tocdepth}{0}

\begin{document}

\null

\vskip-40pt

\title{An operator-valued $T1$ theory for symmetric CZOs}

\author[Hong, Liu, Mei]
{Guixiang Hong, Honghai Liu, Tao Mei}
\address{Guixiang Hong, School of Mathematics and Statistics, Wuhan University, Wuhan 430072, China\\
\emph{E-mail address: guixiang.hong@whu.edu.cn}}

\address{Honghai Liu, School of Mathematics and Information Science,
Henan Polytechnic University, Jiaozuo, Henan, 454003,  China\\
\emph{E-mail address: hhliu@hpu.edu.cn}}

\address{Tao Mei,
Department of Mathematics, Baylor University, USA\\
\emph{E-mail address: tao\_mei@baylor.edu}}

\maketitle

\null

\vskip-45pt

\maketitle

\begin{abstract}
We provide a natural BMO-criterion for the $L_2$-boundedness of Calder\'on-Zygmund operators with operator-valued kernels  satisfying a symmetric property. Our arguments  involve both classical and quantum probability theory.  In the appendix, we give a proof of the $L_2$-boundedness of the commutators $[R_j,b]$ whenever $b$ belongs to the Bourgain's vector-valued BMO space, where $R_j$ is the  $j$-th Riesz transform. A common ingredient    is  the operator-valued Haar multiplier studied by Blasco and Pott.
\end{abstract}

\section{Introduction}

There has been a lot of effort into the  generalization of the classical Calder\'on-Zygmund singular integral theories  to  the  operator-valued (or $d$ by $d$ matrix-valued) setting.
The situation is quite subtle and many  straightforward generalizations are turned out to be wrong.
For example, Pisier and  Harcharras showed (see \cite{Pis98,Har99}) that,  for each $1<p<\infty$,  there exists  a scalar-valued Fourier multiplier $T$ that  is bounded on $L_p({\Bbb R})$ but $T\otimes id_{S_p}$ is not bounded  on $L_p({\Bbb R},S_p)$. Here,  $S_p$ denotes the Schatten-$p$ classes and $L_p({\Bbb R},S_p)$ denotes the space of $S_p$-valued $p$-integrable functions.
Another example is  the dyadic paraproduct $$\pi(b,f)=\sum_{n>0} d_nb E_{n-1}f.$$ Here, $E_n$ denotes the conditional expectation  with respect to the usual dyadic filtration on the real line ${\Bbb R}$ and $d_n$ is the difference  $E_n-E_{n-1}$.
It is well known that $\pi$  maps $L_2({\Bbb R})\times  L_2({\Bbb R})$ to $L_1({\Bbb R})$, and this extends to the vector valued setting that $\pi$ maps $L_2({\Bbb R},\ell_2)\times  L_2({\Bbb R},\ell_2)$ to $L_1({\Bbb R},\ell_1)$. However, $\pi$ fails to map
$L_2({\Bbb R},S_2)\times L_2({\Bbb R},S_2)$ to $L_1({\Bbb R},S_1)$, see \cite{Mei06} and \cite{NPTV02}, \cite{NTV97}.
 This pathological property of $\pi$ prevents a desirable operator-valued $T1$-theory with a natural BMO testing-condition.

The authors notice that this kind of pathological property could be rectified for operators $T$ with a ``symmetric" kernel $K(x,y)$ s.t. $K(x,y)=K(y,x)$, including the Beuling transforms, the Haar multipliers, and the commutator $[R_j, b]$ where $R_j$ is the $j-$Riesz transform.  The main purpose of this article is to formulate  a $T1$ theory with a natural BMO test condition for operator valued Calder\'on-Zygmund operators $T$   satisfying the symmetric property  $(T1)^*=T^*1$.  

In their remarkable work \cite{HW06}, Hyt\"onen and Weiss  already established  an operator-valued $T1$ theory in a quite general setting, i.e. for operator valued singular integral operators on vector valued function space $L_p({\Bbb R}, X)$. Their BMO space seems to be quite complicated and does not contain the space of uniformly bounded $\mathcal B(\ell_2)$-valued functions in the most interesting case $X=\ell_2$. This is necessary because of the bad behavior of operator valued paraproducts mentioned above. The authors hope that this work may complement Hyt\"onen and Weis' work for the case of  symmetric singular integrals. On the other hand, even though strictly speaking the commutator $[R_j, b]$ is not a singular integral operator, we are still able to show its $L_2$-boundedness whenever $b$ satisfies a natural BMO test condition in the same spirit. This result might be essentially known to experts, and we will provide a proof in the Appendix.

\bigskip

A main motivation for the present paper is to investigate noncommutative $T1$ theorem in the semicommutative case, which would provide ideas or insights in searching for $T1$ type theorem in the more general noncommutative setting such as on quantum Euclidean spaces, where a $T1$-theory is in high demand but still missing (see \cite{XXX16},\cite{XXY17},\cite{GJP17},\cite{SZ18},\cite{MSX19}). Let us give an introduction along this research line.
The commutative $T1$ theorem due to David and Journ\'e \cite{DaJo84} is a revolutionary result and finds many applications in classical harmonic analysis \cite{Cal77} \cite{Chr90}. Let $K:\mathbb R^n\times \mathbb R^n\setminus\{(x,x): x\in\mathbb R^n\}\rightarrow \mathbb C$ be a kernel satisfying the standard assumptions:
\begin{align}\label{size1}
|K(x,y)|\lesssim\frac{1}{|x-y|^n},\;\forall x\neq y;
\end{align}
\begin{align}\label{regularity1}
|K(x,y)-K(x',y)|+|K(y,x)-K(y,x')|\lesssim\frac{|x-x'|^\alpha}{|x-y|^{n+\alpha}},
\end{align}
$\forall |x-y|\geq 2|x-x'|$, with some $\alpha\in (0,1]$. Here (and below) $A\lesssim B$ means that there exists an absolute constant $C>0$ such that $A\leq CB$.
A linear operator $T$ initially defined on ``nice" functions is called a Calder\'on-Zygmund operator (CZO) associated with $K$, if $T$ satisfies the kernel representation, for a.e.  $x\notin\supp f$,
\begin{align*}
Tf(x)=\int_{\mathbb R^n}K(x,y)f(y)dy.
\end{align*}

\bigskip

The $T1$ theorem states that $T$ extends to a bounded operator on $L_p(\mathbb R^n)$ for one (or equivalently all) $1<p<\infty$ if and only if
\begin{align}\label{bmo condition1}
T1, T^*1\in \mathrm{BMO}(\mathbb R^n),  \;\mathrm{and}
\end{align}
\begin{align}\label{wbp condition1}
T\; \mathrm{has} \;\mathrm{the} \;\mathrm{Weak}\;\mathrm{ Boundedness}\;\mathrm{ Property}\;\sup_{I\;\mathrm{cube}}\frac{1}{|I|}|\langle 1_I,T1_I\rangle|<\infty.
\end{align}
\bigskip

Along the current research line of noncommutative harmonic analysis, the present paper is devoted to the study of a matrix (operator)-valued $T1$ theorem. More precisely, we are interested in the matrix-valued kernels $K:\mathbb R^n\times \mathbb R^n\setminus\{(x,x): x\in\mathbb R^n\}\rightarrow \mathcal B(\ell_2)$ verifying natural assumptions:
\begin{align}\label{size2}
\|K(x,y)\|_{\mathcal B(\ell_2)}\lesssim\frac{1}{|x-y|^n},\;\forall x\neq y;
\end{align}
\begin{align}\label{regularity2}
\|K(x,y)-K(x',y)\|_{\mathcal B(\ell_2)}+\|K(y,x)-K(y,x')\|_{\mathcal B(\ell_2)}\lesssim\frac{|x-x'|^\alpha}{|x-y|^{n+\alpha}},
\end{align}
$\forall |x-y|\geq 2|x-x'|$. We are   interested in operators  $T$ such that, for all  $ S_{\mathcal B(\ell_2)}$-valued step functions $f$ and a.e. $x\notin\supp f$,
\begin{align}\label{def}
Tf(x)=\int_{\mathbb R^n}K(x,y)f(y)dy=\int_{\mathbb R^n}\sum_{i,j}\big(\sum_kK_{ik}(x,y)f_{kj}(y)\big)\otimes e_{i,j}dy.
\end{align}
Here $S_{\mathcal B(\ell_2)}$ denotes the set of all the elements with finite trace support in $\mathcal B(\ell_2)$. We aim to find a natural BMO condition such as  \eqref{bmo condition1}, \eqref{wbp condition1} such that  $T$ extends to a bounded operator on the noncommutative $L_p$ spaces. Here, the noncommutative $L_p$ spaces are associated to the von Neumann algebra $$\mathcal A=L_\infty(\mathbb R^n)\overline{\otimes} \mathcal B(\ell_2)$$ which consists of all essentially bounded functions $f : \mathbb R^n\rightarrow \mathcal B(\ell_2)$. We refer the reader to \cite{PiXu03} \cite{Xu98} for more information on noncommutative $L_p$ spaces.

The modern development of quantum probability and noncommutative harmonic analysis begun with the seminal paper by Pisier and Xu \cite{PiXu97}, where noncommutative Burkholder-Gundy inequality and Fefferman-Stein duality were established. Later on, many inequalities in classical martingale theory have been transferred into the noncommutative setting \cite{Jun02} \cite{JuXu03} \cite{JuXu08} \cite {Ran02} \cite{Ran07} \cite{HoMe12} \cite{HJP16} \cite{HJP17} \cite{JuPe14} etc. Meanwhile, noncommutative harmonic analysis has gained rapid developments ranging from the noncommutative $H^\infty$-calculus \cite{JLX06, FMS19}, operator-valued harmonic analysis \cite{Mei07} \cite{HoYi13} \cite{HLMP14} \cite{Par09} to Riesz transform/Fourier multipliers on group von Neumann algebras \cite{JMP14} \cite{JMP15} \cite{JuMe10}, hypercontractivity of quantum Markov semigroups \cite{JPPPR12} \cite{JPPP13} \cite{RiXu16} and harmonic analysis on quantum Euclidean spaces/torus \cite{CXY13} \cite{XXY17} \cite{GJP17}.

It worths to point out that the operator-valued (or semi-commutative) harmonic analysis often provides deep insights in harmonic analysis in the general noncommutative setting, and sometimes plays essential role based on the transference principles. For instance, the main ideas of the work \cite{JMP14, CXY13, XXY17, GJP17} are to reduce the problems in their setting to the corresponding problems in the operator-valued setting.

\bigskip

An interesting case is that the functions $f$ are $\ell_2$-valued. This case has been extensively studied in the series of works \cite{TrVo97} \cite{NTV97} \cite{Pet00} \cite{GPTV00} \cite{NPTV02} etc since 97's. In these works, many results in classical harmonic analysis such as weighted norm inequalities, Carleson embedding theorem, Hankel operators, commutators, paraproducts have been extended to the matrix-valued setting. A common character of all these results is that the behavior depends on the dimension of the underlying matrix. For instance, in \cite{NPTV02}, among many other related results, the authors consider the dyadic paraproduct with symbol in noncommutative BMO acting on $\mathbb C^d$-valued functions and show that the bound of the paraproduct operator is of order $O(\log d)$. Since we will not work with noncommutative BMO space $\mathrm{BMO}^{cr}(\mathcal A)$, we refer the reader to \cite{Mei07} for the definition and properties.

More precisely, let $\mathcal D$ be the collection of dyadic intervals in $\mathbb R$. For any dyadic interval $I\in\mathcal D$, let $h_I:={|I|}^{-1/2}(1_{I_+}-1_{I_-})$ be the associated Haar function, where $I_+,I_-$ are left and right halves of the interval $I$. Let $b$ be a $d\times d$-matrix-valued function on $\mathbb R$ and $f$ be a $\mathbb C^d$-valued function on $\mathbb R$, the paraproduct is defined as
\begin{align*}
\pi_b(f):=\sum_{I\in\mathcal D} \mathbb D_I(b)\mathbb E_I(f),
\end{align*}
where $\mathbb D_I(b):=\langle h_I,b\rangle h_I=\int_\mathbb R b(x)h_I(x)\,dx \,h_I$ is a $d\times d$-matrix-valued function on $\mathbb R$ and $\mathbb E_I(f):=\langle \frac{1_I}{|I|},f\rangle 1_I=\mean_If(x)\,dx\, 1_I$ is a $\mathbb C^d$-valued function on $\mathbb R$. In \cite{NPTV02,Mei06}, the authors showed that it may happen
\begin{align}\label{bound of paraproduct1}
\|\pi_b\|_{L_2(\mathbb R;\mathbb C^d)\rightarrow L_2(\mathbb R;\mathbb C^d)}\gtrsim \|b\|_{L_\infty(\mathcal A )}\log d,
\end{align}
This tells us that a naive generalization of classical $T1$ theorem in the semicommutative setting is not true, that is, $T1, T^*1\in \mathrm{BMO}^{cr}(\mathcal A)$ can not guarantee the boundedness of matrix-valued CZOs since the paraproduct is a typical example of perfect dyadic CZOs and $L_\infty(\mathcal A)$ is contained in $\mathrm{BMO}^{cr}(\mathcal A)$ . A CZO on $\mathbb R$ being perfect dyadic means its kernel satisfies the condition (instead of \eqref{regularity2})
\begin{align}\label{perfect}
\|K(x,y)-K(x',y)\|_{\mathcal B(\ell_2)}+\|K(y,x)-K(y,x')\|_{\mathcal B(\ell_2)}=0,
\end{align}
whenever $x,x'\in I$ and $y\in J$ for some disjoint dyadic intervals $I$ and $J$. Perfect dyadic kernels were introduced in \cite{AHMTT02} and include martingale transforns, as well as paraproducts and their adjoints. 

In the remarkable works \cite{Hyt06,HW06}, Hyt\"onen and Weis  have proven  an operator valued $T1$-theorem. However, the BMO-space in their work is a bit artificial and it may not contain $L_\infty$-functions, though this is necessary due to the abnormality of matrix-valued paraproducts.

The first result of the present paper is that under the symmetric assumption $(T1)^*=T^*1$,   the perfect dyadic CZOs $T$ are  bounded on $L_2(\mathcal A)$ provided $T1\in \mathrm{BMO}(\mathcal A,\Sigma_\mathcal A)$, the usual dyadic vector-valued BMO spaces which contains $L_\infty(\mathcal A)$.    Here ``1" means the identity of the algebra $\mathcal A$, and the BMO space $\mathrm{BMO}(\mathcal A,\Sigma_\mathcal A)$ is the dyadic version of the one first studied by Bourgain \cite{Bou86}, whose norm of an operator-valued function $g$ on $\mathbb R$ is defined as
$$\|g\|_{\mathrm{BMO}(\mathcal A,\Sigma_\mathcal A)}=\sup_{I\in\mathcal D}  \Big( \mean_I \|g(x) - g_I\|^2_{\mathcal{B}(\ell_2)} \, dx \Big)^\frac12.$$

On the other hand, providing suitable analogue of \eqref{wbp condition1} for $L_p(\mathcal A)$-boundedness of matrix-valued CZOs when $p\neq2$ is also subtle, since there are some noncommutative martingale transforms with noncommuting coefficients---another type of examples of perfect dyadic CZOs with $T^*1=T1=0$---failing $L_p(\mathcal A)$-boundedness for $p\neq2$, see for instance \cite{Par09}. That implies that a natural Weak Boundedness Property $$\sup_{I\in\mathcal D}\frac{1}{|I|}\|\langle 1_I,T1_I\rangle\|_{\mathcal B(\ell_2)}<\infty$$ can not guarantee the $L_p(\mathcal A)$-boundedness of matrix-valued CZOs for $p\neq2$. In the present paper, we are content with the second best--- showing the boundedness between $L_p(\mathcal A)$ and noncommutative Hardy spaces under the natural Weak Boundedness Property.

\bigskip

Assuming the symmetric condition, we build a weakened form of $T1$ theorem first for the toy model---matrix-valued perfect dyadic CZOs.

\begin{theorem}\label{thm:perfect}
Let $T$ be an operator-valued perfect dyadic CZO satisfying
\begin{align}\label{symmetry1}
\mathrm{Symmetric}\;\mathrm{condition:}\;(T1)^*=T^*1;
\end{align}
\begin{align}\label{bmo condition2}
\mathrm{BMO}\;\mathrm{condition:}\;T1\in \mathrm{BMO}(\mathcal A,\Sigma_\mathcal A);
\end{align}
\begin{align}\label{wbp condition2}
\mathrm{WBP}\;\mathrm{condition:}\;\sup_{I\in\mathcal D}\frac{1}{|I|}\|\langle 1_I,T1_I\rangle\|_{\mathcal B(\ell_2)}<\infty.
\end{align}
Then
  $T$ is bounded on $L_2(\mathcal A)$. Moreover,
\begin{itemize}
\item $T$ is bounded from $L_p(\mathcal A)$ to $\mathrm{H}^c_p(\mathcal A,\Sigma_\mathcal A)$ whenever $2<p<\infty$;
\item $T$ is bounded from $\mathrm{H}^c_p(\mathcal A,\Sigma_\mathcal A)$ to $L_p(\mathcal A)$ whenever $1<p<2$.
\end{itemize}

\end{theorem}


Here $\mathrm{H}^c_p(\mathcal A,\Sigma_\mathcal A)$ is the noncommutative martingale Hardy spaces that we will recall in Section 2. A useful observation in the proof is Lemma \ref{lem:rep of perfect}, which states that dyadic martingale transforms, dyadic paraproducts or their adjoints are essentially the only perfect dyadic CZOs. Then we are reduced to show the boundedness of noncommutative Haar mulitplier---the sum of paraproduct and its adjoint---in Lemma \ref{prop:haar1} where the symmetry is exploited, and the boundedness of noncommutative martingale transform in Lemma \ref{prop:mart1}.

The proof of this toy model is relatively easy but essential for the understanding of our arguments for  (higher-dimensional) general CZOs and commutators.

\bigskip

For {\it continuous CZO}, that is the general singular integrals satisfying \eqref{def} with kernels verifying the standard size and smooth conditions \eqref{size2} \eqref {regularity2}, we  establish a similar result.

\begin{theorem}\label{thm:general}
Let $T$ be a continuous CZO on $\mathbb R^n$ satisfying
\begin{align}\label{symmetry2}
\mathrm{Symmetric}\;\mathrm{condition:}\;(T1)^*=T^*1;
\end{align}
\begin{align}\label{bmo condition3}
\mathrm{BMO}\;\mathrm{condition:}\;T1\in \mathrm{BMO}({\mathbb R^n; \mathcal B(\ell_2)});
\end{align}
\begin{align}\label{wbp condition3}
\mathrm{WBP}\;\mathrm{condition:}\;\sup_{I\;\mathrm{cube}}\frac{1}{|I|}\|\langle 1_I,T1_I\rangle\|_{\mathcal B(\ell_2)}<\infty.
\end{align}
Then $T$ is bounded on $L_2(\mathcal A)$. Moreover,
\begin{itemize}
\item $T$ is bounded from $L_p(\mathcal A)$ to $\mathrm{H}^c_p(\mathbb R^n;\mathcal{B}(\ell_2))$ whenever $2<p<\infty$;
\item $T$ is bounded from $\mathrm{H}^c_p(\mathbb R^n;\mathcal{B}(\ell_2))$ to $L_p(\mathcal A)$ whenever $1<p<2$.
\end{itemize}
\end{theorem}

Here, the BMO and Hardy spaces are the continuous version of the dyadic spaces in the toy model case that we will recall in the body of the proof.
Decompose $T=T_{e}+T_{o}$ as the sum of even and odd parts associated with the kernels $$K_{e}(x,y)=\frac{K(x,y)+K(y,x)}2, \ \ K_{o}(x,y)=\frac{K(x,y)-K(y,x)}2.$$  It is easy to see that $T_e$ satisfies our symmetric assumption $(T_e1)^*=T_e^*1$. We then reduce the $L_2$-boundedness of $T$ to $T_e1\in \mathrm{BMO}(\mathbb R^n;\mathcal B(\ell_2)) $ and the $L_2$-boundedness of $T_{o}$. In particular, together with Remark 1.37 in \cite{HW06}, we get the following corollary.

\begin{corollary}\label{cor:general}
Let $T$ be a continuous CZO on $\mathbb R^n$ satisfying
\par
 \emph{Symmetric condition:}
\begin{align}\label{symmetry3}
 K_o\in L_2(\mathbb R^{2n}; \mathcal B(\ell_2)) \;\mathrm{or}\; T_o1,T_o^*1\in \mathrm{BMO}(\mathbb R^n;S_q(\ell_2)) \;1<q<\infty;
\end{align}
\begin{align}\label{bmo condition4}
\mathrm{BMO}\;\mathrm{condition:}\; T_e1\in \mathrm{BMO}(\mathbb R^n; {\mathcal B(\ell_2)}) ;
\end{align}
\begin{align}\label{wbp condition4}
\mathrm{WBP}\;\mathrm{condition:}\;\sup_{I\;\mathrm{cube}}\frac{1}{|I|}\|\langle 1_I,T1_I\rangle\|_{\mathcal B(\ell_2)}<\infty.
\end{align}
Then $T$ is bounded on $L_2(\mathcal A)$.
\end{corollary}


Theorem \ref{thm:perfect}, \ref{thm:general}  and Corollary \ref{cor:general}  hold for general operator-valued functions, e.g. replacing $\mathcal B(\ell_2)$ by any semifinite von Neumann algebra $\mathcal M$. Our proof will be written in this general framework.

As in classical harmonic analysis  \cite{AHMTT02}, from the result in the dyadic setting--Theorem \ref{thm:perfect}, it is usually not difficult to guess similar result---Theorem \ref{thm:general}---in the continuous setting. In scalar-valued harmonic analysis, we can realize this passage from the dyadic setting to the continuous one by dealing with issues such as rapidly decreasing tails or using the Vitali covering lemma. In the case of vector-valued harmonic analysis, this passage requires deep understanding on the connection between martingale theory and harmonic analysis as done in \cite{Bou86} \cite{Bur86} \cite{Fig90} \cite{HoMa16} \cite{HoMa17} \cite{Xu98} \cite{MTX06} etc. In noncommutative harmonic analysis,  in addition to the idea or the techniques developed in vector-valued theory,  new idea, techniques or tools developed in noncommutative analysis are usually needed to realize this passage such as in \cite{Mei07}, \cite{HoYi13}. In the present paper, the main idea or technique from vector-valued theory we need is the method of random dyadic cubes firstly introduced in \cite{NTV03}, later modified in \cite{Hyt12} \cite{Hyt10}.

\bigskip

We will show Theorem \ref{thm:perfect} and \ref{thm:general} in Section 2 and 3 respectively. The definitions of BMO spaces and Hardy spaces as well as the method of random dyadic cubes will be properly recalled in the body of the paper. In the Appendix, we will show that the commutator $[R_j,b]$ is $L_2$-bounded whenever $b$ belongs to Bourgain's vector-valued BMO space $\mathrm{BMO}(\mathbb R^n;\mathcal M)$. This result might be essentially known to experts, but we do not find it in any literature, and thus we put it in the Appendix.

\section{Perfect dyadic CZOs: proof of Theorem \ref{thm:perfect}}

Let $\M$ be a semifinite von Neumann algebra equipped with a normal semifinite faithful trace $\tau$. Consider the algebra of essentially bounded functions $\mathbb R \to \M$ equipped with the n.s.f. trace $$\varphi(f) = \int_{\mathbb R} \tau(f(x)) \, dx.$$ Its weak-operator closure is a von Neumann algebra $\A$. If $1 \le p \le \infty$, we write $L_p(\M)$ and $L_p(\A)$ for the noncommutative $L_p$ spaces associated to the pairs $(\M,\tau)$ and $(\A,\varphi)$. The set of all the elements with finite trace support in $\M$ is written as $S_\M$.
The set of dyadic intervals in $\mathbb R$ is denoted by $\mathcal D$ and we use $\mathcal D_k$ for the $k$-th generation, formed by intervals $I$ with side length $\ell(I) = 2^{-k}$.
We consider the associated filtration $(L_{\infty}(\mathcal D_k)\overline{\otimes}\mathcal M)_{k\in\mathbb Z}$ of $\mathcal A$, which will be simplified as $\Sigma_\mathcal A=(\A_k)_{k\in\mathbb Z}$. Let $\mathbb E_k$ and $\mathbb D_k$ denote the corresponding conditional expectations and martingale difference operators.

\subsection{Two auxiliary results.}

In the present section, we first show two auxiliary results with respect to the following two kinds of operators:
\begin{itemize}
\item{Noncommuting martingale transforms} $$M_\xi f = \sum_{k\in\mathbb Z} \xi_{k-1} \mathbb D_k(f),$$

\item{Haar multipliers with noncommuting symbol} $$\Lambda_b(f) = \sum_{k\in\mathbb Z} \mathbb D_k(b) \mathbb{E}_{k}(f).$$
\end{itemize}
Here $\xi_{k} \in \A_{k}$ is an adapted sequence. Of course, the symbols $\xi$ and $b$ do not necessarily commute with the function.  Our arguments on the operator-valued $\Lambda_b$ follow the ideas from \cite{BlPo08, BlPo10} and \cite{Mei07}.

 Let us first recall more definitions.

\bigskip

\noindent \textbf{Noncommutative martingale Hardy spaces.} Let
$1\leq p<\infty$. The column Hardy space $\mathrm{H}^c_p(\A,\Sigma_\mathcal A)$ is defined to be the completion of all finite $L_p$-martingales under the norm
$$\|f\|_{\mathrm{H}^c_p(\A,\Sigma_\mathcal A)}:=\|\big(\sum_{k\in\mathbb Z}\mathbb D_kf^*\mathbb D_kf\big)^{1/2}\|_p.$$
Taking adjoint---so that the $*$ switches from left to right--- we find the row-Hardy space norm and the space. The noncommutative Hardy space $\mathrm{H}_p(\mathcal A, \Sigma_\mathcal A)$, defined through column and row spaces differently for $1\leq p<2$ and $p>2$, was introduced in \cite{PiXu97}. In the same paper, the authors also introduced noncommutative martingale BMO spaces, and show the noncommutative Burkholder-Gundy inequality and Fefferman-Stein duality.  According to \cite{Mus03} it has the expected interpolation behavior in the scale of noncommutative $L_p$ spaces. Since we will not use these result in the present paper, we omit the details.

\bigskip

\noindent \textbf{Vector-valued BMO spaces.} The $\M$-valued martingale BMO space $\mathrm {BMO}(\mathcal A, \Sigma_\mathcal A)$ is defined to the set of $\mathcal M$-valued locally integrable functions with norm
$$\|f\|_{\mathrm {BMO}(\mathcal A, \Sigma_\mathcal A)}:=\sup_{k\in\mathbb Z}\|\mathbb{E}_k\|f-\mathbb E_{k-1}f\|^2_{\M}\|^{1/2}_{L_\infty(\mathbb R)}.$$
This space  is related to the vector-valued Hardy space $\mathrm{H}^m_1(\mathcal A, \Sigma_\mathcal A)$ whose norm is defined as
$$\|f\|_{\mathrm{H}^m_1(\mathcal A, \Sigma_\mathcal A)}:=\|\sup_{k\in\mathbb Z}\|\mathbb {E}_kf\|_{L_1(\mathcal M)}\|_{L_1(\mathbb R)}.$$
In fact, Bourgain \cite{Bou86} and Garcia-Cuerva proved independently that $\mathrm {BMO}(\mathcal A, \Sigma_\mathcal A)$ embeds continuously into the dual of $\mathrm{H}^m_1(\mathcal A, \Sigma_\mathcal A)$. That is
\begin{align}\label{dual of vector}
|\varphi(f^*g)|\lesssim \|f\|_{\mathrm {BMO}(\mathcal A, \Sigma_\mathcal A)}\|g\|_{\mathrm{H}^m_1(\mathcal A, \Sigma_\mathcal A)}.
\end{align}
We also need the following Doob's inequality for $L_p(\M)$-valued function: For all $1<p\leq\infty$ and $f\in L_p(\mathcal A)$
\begin{align}\label{doob vector}
\|\sup_{k\in\mathbb Z}\|\mathbb {E}_kf\|_{L_p(\mathcal M)}\|_{L_p(\mathbb R)}\lesssim \|f\|_{L_p(\A)}.
\end{align}

\bigskip

\begin{proposition}\label{prop:mart1}
If $\sup_k\|\xi_k\|_\A<\infty$, then
\begin{itemize}
\item $M_\xi$ is bounded from $L_p(\mathcal A)$ to $\mathrm{H}^c_p(\A,\Sigma_\A)$ whenever $2\leq p<\infty$;
\item $M_\xi$ is bounded from $\mathrm{H}^c_p(\A,\Sigma_\A)$ to $L_p(\mathcal A)$ whenever $1<p\leq2$.
\end{itemize}
\end{proposition}

\begin{proof}
We only give the proof of the case $2\leq p<\infty$, another case can be shown similarly. Let $f\in L_p(\mathcal A)$. Using the fact $a^*ca\leq a^*a\|c\|_\infty$ for any $a\in\A$ and $c\in\A^+$ and $\xi_k\in\A_k$, it is easy to check
\begin{align*}
\|M_\xi(f)\|_{\mathrm{H}^c_p(\A,\Sigma_\A)}&=\|\big(\sum_{k\in\mathbb Z}\mathbb D_kf^*\xi^*_{k-1}\xi_{k-1}\mathbb D_kf\big)^{1/2}\|_p\\
&\leq \sup_k\|\xi_k\|_{L_\infty(\A)}\|\big(\sum_{k\in\mathbb Z}\mathbb D_kf^*\mathbb D_kf\big)^{1/2}\|_p\lesssim\|f\|_{L_p(\A)}.
\end{align*}
We have used the H\"older inequality and Burkholder-Gundy inequality in the inequalities.
\end{proof}

\begin{proposition}\label{prop:haar1}
If $b\in \mathrm {BMO}(\mathcal A, \Sigma_\mathcal A)$, then we have
\begin{itemize}
\item $\Lambda_b$ is bounded from $L_p(\mathcal A)$ to $\mathrm{H}^c_p(\A,\Sigma_\A)$ whenever $2\leq p<\infty$;
\item $\Lambda_b$ is bounded from $\mathrm{H}^c_p(\A,\Sigma_\A)$ to $L_p(\mathcal A)$ whenever $1<p\leq2$.
\end{itemize}
\end{proposition}

\begin{proof}
We only provide the proof of the case $2\leq p<\infty$, since another case can be shown similarly. Let $q$ be the conjugate index of $p$. Let $f\in L_p(\A)$, and $g\in \mathrm{H}^c_q(\A,\Sigma_\A)$. By duality, it suffices to show
$$|\varphi(\Lambda_b(f)g^*)|\lesssim \|f\|_{L_p(\A)}\|g\|_{\mathrm{H}^c_q(\A,\Sigma_\A)}.$$
 Using the assumption that $\mathbb D_k(b)\mathbb D_k(f)\in\A_{k-1}$ for each $k$, we have
\begin{align*}
|\varphi(\Lambda_b(f)g^*)|&=|\sum_{k\in\mathbb Z} \varphi(\mathbb D_k(b) \mathbb{E}_{k}(f)g^*)|\\
&=|\sum_{k\in\mathbb Z} \varphi(\mathbb D_k(b) \mathbb{E}_{k-1}(f)g^*+\mathbb D_k(b) \mathbb{D}_{k}(f)g^*)|\\
&=|\sum_{k\in\mathbb Z} \varphi(\mathbb D_k(b) \mathbb{E}_{k-1}(f)\mathbb{D}_k(g^*)+\mathbb D_k(b) \mathbb{D}_{k}(f)\mathbb E_{k-1}(g^*))|\\
&=| \varphi(b \sum_{k\in\mathbb Z}(\mathbb{E}_{k-1}(f)\mathbb{D}_k(g^*)+\mathbb{D}_{k}(f)\mathbb E_{k-1}(g^*)))|.
\end{align*}
Hence by duality between vector-valued BMO space and Hardy space, we have
\begin{align*}
&|\varphi(\Lambda_b(f)g^*)|\\
&\lesssim \|b\|_{\mathrm {BMO}(\mathcal A, \Sigma_\mathcal A)}\|\sum_{k\in\mathbb Z}\mathbb{E}_{k-1}(f)\mathbb{D}_k(g^*)+\sum_{k\in\mathbb Z}\mathbb{D}_{k}(f)\mathbb E_{k-1}(g^*)\|_{\mathrm {H}^m_1(\mathcal A, \Sigma_\mathcal A)}\\
&= \|b\|_{\mathrm {BMO}(\mathcal A, \Sigma_\mathcal A)}\int_{\mathbb R}\sup_{\ell\in\mathbb Z}\|\sum^\ell_{k=-\infty}\mathbb{E}_{k-1}(f)\mathbb{D}_k(g^*)+\sum^\ell_{k=-\infty}\mathbb{D}_{k}(f)\mathbb E_{k-1}(g^*)\|_{L_1(\M)}dx.
\end{align*}
Using the identity for each $\ell\in\mathbb Z$
\begin{eqnarray*}
&\sum^\ell_{k=-\infty}\mathbb{E}_{k-1}(f)\mathbb{D}_k(g^*)+\sum^\ell_{k=-\infty}\mathbb{D}_{k}(f)\mathbb E_{k-1}(g^*)\\
&=\mathbb E_\ell(f)\mathbb {E}_\ell(g^*)-\sum^\ell_{k=-\infty}\mathbb D_k(f)\mathbb {D}_k(g^*),
\end{eqnarray*}
we are reduced to show
\begin{align}\label{key1}
\int_{\mathbb R}\sup_{\ell\in\mathbb Z}\|\mathbb E_\ell(f)\mathbb {E}_\ell(g^*)\|_{L_1(\M)}dx\lesssim\|f\|_{L_p(\A)}\|g\|_{\mathrm{H}^c_q(\A,\Sigma_\A)}
\end{align}
and
\begin{align}\label{key2}
\int_{\mathbb R}\sup_{\ell\in\mathbb Z}\|\sum^\ell_{k=-\infty}\mathbb D_k(f)\mathbb {D}_k(g^*)\|_{L_1(\M)}dx\lesssim\|f\|_{L_p(\A)}\|g\|_{\mathrm{H}^c_q(\A,\Sigma_\A)}.
\end{align}
The first estimate is relatively easy to handle. Using twice the H\"older inequalities and vector-valued Doob's inequality \eqref{doob vector}, the left hand side of  \eqref{key1} is controlled by
\begin{align*}
&\leq  (\int_{\mathbb R}\sup_{\ell\in\mathbb Z}\|\mathbb E_\ell(f)\|^p_{L_p(\M)}dx)^{1/p}(\int_{\mathbb R}\sup_{\ell\in\mathbb Z}\|\mathbb E_\ell(g)\|^q_{L_q(\M)}dx)^{1/q}\\
&\lesssim \|f\|_{L_{p}(\A)}\|g\|_{L_q(\A)}\lesssim \|f\|_{L_{p}(\A)}\|g\|_{\mathrm{H}^c_q(\A,\Sigma_\A)},
\end{align*}
where we used noncommutative Burkholder-Gundy inequality for $q\leq2$ in the last inequality.

To show the second estimate \eqref{key2}, we only need to show for any $\ell$
$$\|\sum^\ell_{k=-\infty}\mathbb D_k(f)\mathbb {D}_k(g^*)\|_{L_1(\M)}\leq \|(\sum_{k\in\mathbb Z}|\mathbb D_k(f)|^2)^{1/2}\|_{L_p(\M)}\|(\sum_{k\in\mathbb Z}|\mathbb D_k(g)|^2)^{1/2}\|_{L_q(\M)},$$
since then we can follow similar arguments as in the \eqref{key1}.
By duality and the H\"older inequality,
\begin{align*}
&\|\sum^\ell_{k=-\infty}\mathbb D_k(f)\mathbb {D}_k(g^*)\|_{L_1(\M)}\\
&=\sup_{u,\;\|u\|_\M\leq1}|\tau(u\sum^\ell_{k=-\infty}\mathbb D_k(f)\mathbb {D}_k(g^*))|=\sup_{u,\;\|u\|_\M\leq1}|\tau(\sum^\ell_{k=-\infty}\mathbb {D}_k(g^*)(u\mathbb D_k(f)))|\\
&=\sup_{u,\;\|u\|_\M\leq1}|\tau\otimes tr((\sum^\ell_{k=-\infty}\mathbb {D}_k(g^*)\otimes e_{1k})(\sum^\ell_{k=-\infty}u\mathbb D_k(f)\otimes e_{k1}))|\\
&\leq \sup_{u,\;\|u\|_\M\leq1}\|\sum^\ell_{k=-\infty}\mathbb {D}_k(g^*)\otimes e_{1k}\|_{L_q(\M\overline{\otimes}\mathcal B(\ell_2))}\|\sum^\ell_{k=-\infty}u\mathbb D_k(f)\otimes e_{k1}\|_{L_p(\M\overline{\otimes}\mathcal B(\ell_2))}\\
&\leq\|(\sum_{k\in\mathbb Z}|\mathbb D_k(f)|^2)^{1/2}\|_{L_p(\M)}\|(\sum_{k\in\mathbb Z}|\mathbb D_k(g)|^2)^{1/2}\|_{L_q(\M)}.
\end{align*}
\end{proof}


\subsection{Representation of perfect dyadic CZOs.}  To the best of our knowledge, the notion of perfect dyadic CZO was rigorously defined for the first time in \cite{AHMTT02}. Classical perfect dyadic CZOs include Haar multipliers/martingale transforms and dyadic paraproducts or their adjoints. In the cited paper, they also show that these operators and their combinations are the only perfect dyadic CZOs. That is, any operator-valued perfect dyadic CZO is a sum of one noncommutative dyadic martingale transform, one noncommutative dyadic paraproduct and its adjoint.

Let us fix some notations in the present section. If $f: \R \to \M$ is integrable on $I \in \mathcal D$, we set the average $$f_I = \mean_I f(x) \, dx.$$ If $1 \le p \le \infty$ and $f \in L_p(\A)$
$$\mathbb{E}_k(f):= \sum_{I\in\mathcal D_k} \mathbb E_I(f) := \sum_{I\in\mathcal D_k}  f_I 1_I, \;\; \mathbb D_k(f) := \sum_{I \in \mathcal D_{k-1}} \mathbb D_I(f) := \sum_{I \in \mathcal D_{k-1}}\langle h_I,f\rangle h_I,$$
where $h_I:=|I|^{-1/2}(1_{I_+}-1_{I_-})$ is the Haar function and $\langle \cdot,\cdot\rangle$ denotes the operator-valued inner product anti-linear in first coordinate. We will use $\langle\langle \cdot,\cdot\rangle\rangle$ to denote the inner product in $L_2(\mathcal A)$ anti-linear in first coordinate.

\begin{lemma}\label{lem:rep of perfect}
Let $T$ be an operator-valued perfect dyadic CZO. Then for $f,g\in\mathcal S(\mathbb R)\otimes S_\M$,
\begin{align}\label{rep of perfect}
\nonumber \langle\langle g,T(f)\rangle\rangle=&\llangle g,\sum_{I\in\mathcal D}\langle h_I,T(h_I)\rangle\langle h_I,f\rangle h_I\rrangle\\
\;\;\;\;\;\;\;\;\;\;\;\;\;\;\;\;\;&+\llangle g,\sum_{I\in\mathcal D}\mathbb D_I((T^*1)^*)\mathbb D_{I}(f)\rrangle+\llangle g,\sum_{I\in\mathcal D}\mathbb D_I(T1)\mathbb E_{I}(f)\rrangle.
\end{align}
\end{lemma}

This representation \eqref{rep of perfect} has been essentially verified in \cite{AHMTT02} using the language of wave package. Here, we prefer to give a proof using an alternate approach due to Figiel \cite{Fig90}, which motivates us to deduce a similar representation formula for general matrix-valued Calder\'on-Zygmund operators in the next section.

\begin{proof}
Without loss of generality, we can assume that both $f$ and $g$ are of the form $h\otimes m$ with $h$ being scalar-valued function and $m$ being an operator. Then the convergence of $\mathbb E_k(h)$ to $h$ as $k\rightarrow \infty$ and to $0$ as $k\rightarrow -\infty$ (both a.e. and in $L_p(\mathbb R)$) leads to Figiel's representation of $T$ as the telescopic series
\begin{align*}
\langle\langle g,Tf\rangle\rangle&=\sum^\infty_{k=-\infty} (\langle\langle \mathbb E_kg, T\mathbb E_kf\rangle\rangle-\langle\langle \mathbb E_{k-1}g,T\mathbb E_{k-1}f\rangle\rangle)\\
&=\sum^\infty_{k=-\infty} (\langle\langle \mathbb D_kg,T\mathbb D_kf,\rangle\rangle+\langle\langle \mathbb E_{k-1}g,T\mathbb D_{k}f\rangle\rangle+\langle\langle \mathbb D_{k}g,T\mathbb E_{k-1}f\rangle\rangle)\\
&:=A+B+C,
\end{align*}
where, upon expanding in terms of the Haar functions,
$$A=\sum_{m\in\mathbb Z}\sum_{I\in\mathcal D}\langle\langle g,\langle h_{I\dot{+}m},Th_I\rangle\langle h_I,f\rangle h_{I\dot{+}m}\rangle\rangle,$$
\begin{align*}
B=\sum_{m\in\mathbb Z}\sum_{I\in\mathcal D}\langle\langle g,\langle \frac{1_{I\dot{+}m}}{|I\dot{+}m|},Th_I\rangle\langle h_I,f\rangle \frac{1_{{I\dot{+}m}}}{|I_{I\dot{+}m}|}\rangle\rangle,
\end{align*}
and
\begin{align*}
C=\sum_{m\in\mathbb Z}\sum_{I\in\mathcal D}\langle\langle g,\langle h_{I},T\frac{1_{{I\dot{+}m}}}{|I_{I\dot{+}m}|}\rangle f_{I\dot{+}m}h_{I}\rangle\rangle.
\end{align*}
Here $I\dot{+}m:=I+\ell(I)m$ is the translation of a dyadic interval $I$ by $m\in\mathbb Z$ times its sidelength $\ell(I)$.
Now by the perfect property of the kernel \eqref{perfect}---$T1_J$ is supported in $J$ for any dyadic interval $J$, we see that only the term $m=0$ in the summation contributes. Then observing that $|h_I|^2=1_I/|I|$, we see clearly
$$B=\llangle g,\sum_{I\in\mathcal D}\mathbb D_I((T^*1)^*)\mathbb D_{I}(f)\rrangle$$
and
$$C=\llangle g,\sum_{I\in\mathcal D}\mathbb D_I(T1)\mathbb E_{I}(f)\rrangle$$
finishing the proof.
\end{proof}

\subsection{Proof of Theorem \ref{thm:perfect}.} From the representation \eqref{rep of perfect}, using the symmetric condition \eqref{symmetry1}, we clearly have $T(f)=M_{\xi}(f)+\Lambda_b(f)$ with
\begin{align*}
\xi_{k}=\sum_{I\in\mathcal D_k}\langle h_I,T(h_I)\rangle\; 1_I,\quad \; b=T1.
\end{align*}
Then observing that $\|\cdot\|_{\mathrm{H}^c_2(\A,\Sigma_\A)}=\|\cdot\|_{L_2(\mathcal A)}$, we finish the proof using Proposition \eqref{prop:mart1} and \eqref{prop:haar1} since WBP condition \eqref{wbp condition2} ensures $\sup_k\|\xi_k\|_\A<\infty$, while BMO condition \eqref{bmo condition2} ensures $b\in\mathrm{BMO}(\A,\Sigma_\A)$.



\section{General CZOs: proof of Theorem \ref{thm:general}}
As in the proof of classical $T1$ theorem, the most difficult part of Theorem \ref{thm:general} is the case $p=2$, and other cases will follow by standard arguments. We will summarize the proof at the end of this section. For the case $p=2$, as mentioned in the Introduction, we will use the method of random dyadic cubes first introduced in \cite{NTV03}, later modified in \cite{Hyt10}. For the sake of completeness, let us recall necessary details of this approach in the present paper. We refer the reader to the previously cited papers for more information.

\subsection{Radom dyadic system}
Let $\mathcal D^0:=\bigcup_{j\in\mathbb Z}\mathcal D^0_j$, $\mathcal D^0_j:=\{2^{-j}([0,1)^n+m):\;m\in\mathbb Z^n\}$ be the standard system of dyadic cubes---the one in the previous section when $n=1$. For every $\beta=(\beta_j)_{j\in \mathbb Z}\in(\{0,1\}^n)^{\mathbb Z}$, consider the dyadic system $\mathcal D^\beta=\{I+\beta:I\in\mathcal D^0\}$ where $I+\beta:=I+\sum_{i:2^{-i}<\ell(I)}2^{-i}\beta_i$.

The product probability $\mathbb P_{\beta}$ on $(\{0,1\}^n)^\mathbb Z$ induces a probability on the family of all dyadic systems $\mathcal D^\beta$. Consider for a moment a fixed dyadic system $\mathcal D=\mathcal D^\beta$ for some $\beta$. A cube $I\in\mathcal D$ is called `bad' (with parameters $r\in\mathbb Z_+$ and $\gamma\in(0,1)$) if there holds
$$\mathrm{dist}(I,J^c)\leq\ell(I)^{\gamma}\ell(J)^{1-\gamma} \quad\mathrm{for}\;\mathrm{some}\;J=I^{(k)},\quad k\geq r,$$
where $I^{(k)}$ denotes the $k$-th dyadic ancestor of $I$. Otherwise, $I$ is said to be `good'.

Fixing a $I\in\mathcal D^0$, consider the random event that its shift $I+\beta$ is bad in $\mathcal D^\beta$. Because of the symmetry it is obvious that the probability $\mathbb P_\beta(I+\beta\;\mathrm{is}\;\mathrm{bad})$ is independent of the cube $I$, and we denote it by $\pi_{\mathrm{bad}}$; similarly one defines $\pi_{\mathrm{good}}=1-\pi_{\mathrm{bad}}$. The only thing that is needed about this number in the present paper as in \cite{NTV03} \cite{Hyt12} is that $\pi_{\mathrm{bad}}<1$, and hence $\pi_{\mathrm{good}}>0$, as soon as $r$ is chosen sufficiently large. We henceforth consider the parameters $\gamma$ and $r$ being fixed in such a way.

Note that
$$\pi_{\mathrm{good}}=\mathbb P_{\beta}(I+\beta\; \mathrm{is} \;\mathrm{good})=\mathbb E_{\beta}1_{\mathrm{good}}(I+\beta)$$
which is independent of the particular cube $I$.
Then as in \cite{Hyt10}, using the fact that the event that $I+\beta$ is good is independent of the position of the cube $I+\beta$, hence of the function $\phi(I+\beta)$, for $\phi(I)$ defined on all the cubes, we have
\begin{align}\label{key iden of prob}
\pi_{\mathrm{good}}\mathbb E_{\beta}\sum_{I\in\mathcal D^\beta}\phi(I)&=\sum_{I\in\mathcal D^0}\mathbb E_{\beta}1_{\mathrm{good}}(I+\beta)\mathbb E_{\beta}\phi(I+\beta)\\
\nonumber &=\sum_{I\in\mathcal D^0}\mathbb E_{\beta}(1_{\mathrm{good}}(I+\beta)\phi(I+\beta))=\mathbb E_{\beta}\sum_{I\in\mathcal D^\beta_{\mathrm{good}}}\phi(I).
\end{align}
This identity is the only thing from the probabilistic approach that we will use in the present paper.

\subsection{Representation of general CZOs.}
Fix a $\beta\in(\{0,1\}^n)^\mathbb Z$. For $\mathcal D=\mathcal D^{\beta}$, let $\mathbb E_k$ be the associated conditional expectation with respect to $\mathcal D_k$, and $\mathbb D_k:=\mathbb E_k-\mathbb E_{k-1}$. These operators can be represented by the Haar functions $h^{\theta}_I$, $\theta\in\{0,1\}^n$, which is defined as follows: When $n=1$,
$$h^0_I:=|I|^{\frac12}1_I\;, h^1_{I}:=|I|^{\frac12}(1_{I_+}-1_-);$$
When $n\geq2$,
$$h^{\theta}_I(x):=h^{(\theta_1,,\dotsm,\theta_n)}_{I_1\times \dotsm\times I_n}(x_1,\dotsm,x_n)=\prod^n_{i=1}h^{\theta_i}_{I_i}(x_i).$$
Then
$$\mathbb E_k(f)=\sum_{I\in\mathcal D_k}h_I^0\langle h^0_I,f\rangle,\;\mathbb D_k(f)=\sum_{I\in\mathcal D_{k-1}}\sum_{\theta\in\{0,1\}^n\setminus\{0\}}h_I^\theta\langle h^\theta_I,f\rangle.$$
The translation of a dyadic cube $I$ by $m\in\mathbb Z^n$ times its sidelength $\ell(I)$, is defined similarly as $I\dot{+}m:=I+m\ell(I)$.

As in Lemma \ref{lem:rep of perfect}, we also have Figiel's representation of an operator-valued Calder\'on-Zygmund operator. Let  $f,g\in\mathcal S(\mathbb R^n)\otimes S_\M$.

\begin{align}\label{1}
\langle\langle g,Tf\rangle\rangle&=\sum^\infty_{k=-\infty} (\langle\langle \mathbb D_kg,T\mathbb D_kf\rangle\rangle+\langle\langle \mathbb E_{k-1}g,T\mathbb D_{k}f\rangle\rangle+\langle\langle \mathbb D_{k}g,T\mathbb E_{k-1}f\rangle\rangle)\\
\nonumber&=:A+B+C,
\end{align}
where
$$A=\sum_{\eta,\theta\in\{0,1\}^n\setminus\{0\}}\sum_{m\in\mathbb Z^n}\sum_{I\in\mathcal D}\langle\langle g,\langle h^{\eta}_{I\dot{+}m},Th^\theta_I\rangle\langle h^\theta_I,f\rangle h^\eta_{I\dot{+}m}\rangle\rangle;$$
\begin{align*}
B&=\sum_{\theta\in\{0,1\}^n\setminus\{0\}}\sum_{m\in\mathbb Z^n}\sum_{I\in\mathcal D}\langle\langle g,\langle h^0_{I\dot{+}m},Th^\theta_I\rangle\langle h^\theta_I,f\rangle h^0_{I\dot{+}m}\rangle\rangle\\
&=\sum_{\theta\in\{0,1\}^n\setminus\{0\}}\sum_{m\in\mathbb Z^n}\sum_{I\in\mathcal D}\langle\langle g,\langle h^0_{I\dot{+}m},Th^\theta_I\rangle\langle h^\theta_I,f\rangle (h^0_{I\dot{+}m}-h^0_I)\rangle\rangle\\
&+\sum_{\theta\in\{0,1\}^n\setminus\{0\}}\sum_{I\in\mathcal D}\langle\langle g,\langle h^\theta_I,(T^*1)^*\rangle\langle h^\theta_I,f\rangle 1_{I}/|I|\rangle\rangle=:B^0+P;
\end{align*}
and
\begin{align*}
C&=\sum_{\theta\in\{0,1\}^n\setminus\{0\}}\sum_{m\in\mathbb Z^n}\sum_{I\in\mathcal D}\langle\langle g,\langle h^\theta_{I},Th^0_{I\dot{+}m}\rangle\langle h^0_{I\dot{+}m},f\rangle h^\theta_{I}\rangle\rangle\\
&=\sum_{\theta\in\{0,1\}^n\setminus\{0\}}\sum_{m\in\mathbb Z^n}\sum_{I\in\mathcal D}\langle\langle g,\langle h^\theta_{I},Th^0_{I\dot{+}m}\rangle\langle h^0_{I\dot{+}m}-h^0_I,f\rangle h^\theta_{I}\rangle\rangle\\
&+\sum_{\theta\in\{0,1\}^n\setminus\{0\}}\sum_{I\in\mathcal D}\langle\langle g,\langle h^\theta_{I},T1\rangle\langle 1_{I}/|I|,f\rangle h^\theta_{I}\rangle\rangle=:C^0+Q.
\end{align*}

Taking integral $\mathbb E_\beta$ on both sides of identity \eqref{1}, and then using the identity \eqref{key iden of prob}, we get
$$\langle\langle g,Tf\rangle\rangle=\frac{1}{\pi_{\mathrm{good}}}\mathbb E_\beta(A_{\mathrm{good}}+B^0_{\mathrm{good}}+C^0_{\mathrm{good}})+\mathbb E_\beta(P+Q),$$
where for instance
$$A_{\mathrm{good}}=A^{\beta}_{\mathrm{good}}=\sum_{\eta,\theta\in\{0,1\}^n\setminus\{0\}}\sum_{m\in\mathbb Z^n}\sum_{I\in\mathcal D^\beta_{\mathrm{good}}}\langle\langle g,\langle h^{\eta}_{I\dot{+}m},Th^\theta_I\rangle\langle h^\theta_I,f\rangle h^\eta_{I\dot{+}m}\rangle\rangle.$$

The desired estimate
$$|\langle\langle g,Tf\rangle\rangle|\lesssim\|f\|_{L_2(\mathcal A)}\|g\|_{L_2(\mathcal A)}$$
is reduced to the corresponding uniform estimate (in $\beta$) for $A_{\mathrm{good}}$, $B^0_{\mathrm{good}}$, $C^0_{\mathrm{good}}$ and $P+Q$.

\bigskip

\noindent{\bf Estimate of $A_{\mathrm{good}}$}. This term can be estimated directly since $\{h^\eta_{I\dot{+}m}\}_{I\in\mathcal D}$ form a martingale difference sequence for fixed $\eta,m$.
\begin{align*}|A_{\mathrm{good}}|&\leq \sum_{\eta,\theta\in\{0,1\}^n\setminus\{0\}}\sum_{m\in\mathbb Z^n}|\langle\langle g,\sum_{I\in\mathcal D_{\mathrm{good}}}\langle h^{\eta}_{I\dot{+}m},Th^\theta_I\rangle\langle h^\theta_I,f\rangle h^\eta_{I\dot{+}m}\rangle\rangle|\\
&\leq \|g\|_{L_2(\mathcal A)}\sum_{\eta,\theta\in\{0,1\}^n\setminus\{0\}}\sum_{m\in\mathbb Z^n}\|\sum_{I\in\mathcal D_{\mathrm{good}}}\langle h^{\eta}_{I\dot{+}m},Th^\theta_I\rangle\langle h^\theta_I,f\rangle h^\eta_{I\dot{+}m}\|_{L_2(\mathcal A)}\\
&\leq \|g\|_{L_2(\mathcal A)}\sum_{\eta,\theta\in\{0,1\}^n\setminus\{0\}}\sum_{m\in\mathbb Z^n}\sup_{I \;\mathrm{cubes}}\|\langle h^{\eta}_{I\dot{+}m},Th^\theta_I\rangle\|_{\M}\big(\tau(\sum_{I\in\mathcal D_{\mathrm{good}}}|\langle h^\theta_I,f\rangle|^2)\big)^{\frac12}\\
&\lesssim \|g\|_{L_2(\mathcal A)}\sum_{\eta,\theta\in\{0,1\}^n\setminus\{0\}}\sum_{m\in\mathbb Z^n}(1+|m|)^{-n-\alpha}\|f\|_{L_2(\mathcal A)}\lesssim \|g\|_{L_2(\mathcal A)}\|f\|_{L_2(\mathcal A)}.
\end{align*}
Here we used the fact
\begin{align}\label{decay estimate}
\sup_{\eta,\theta\in\{0,1\}^n}\sup_{I \;\mathrm{cubes}}\|\langle h^{\eta}_{I\dot{+}m},Th^\theta_I\|_{\M}\lesssim (1+|m|)^{-n-\alpha}
\end{align}
which was essentially observed by Figiel following from the size condition \eqref{size2}, the smooth condition \eqref{regularity2} and the Weak Boundedness Condition \eqref{wbp condition3}.

\bigskip

\noindent{\bf Estimates of $B^0_{\mathrm{good}}$}. This term can not be estimated directly since $h^0_{I\dot{+}m}-h^0_I$ do not form a martingale difference sequence when $I$ runs over all elements in $\mathcal D_{\mathrm{good}}$, but can be achieved when $I$ runs over elements in some subcollections which partition $\mathcal D_{\mathrm{good}}$ as in \cite{Hyt10}.

For each $m$, let $M=M(m):=\max\{r,\lceil(1-\gamma)^{-1}\log^+_2|m|\rceil\}$. Let then $a(I):=\log_2\ell(I)$ mod $M+1$, and define $b(I)$ to be alternatingly $0$ and $1$ along each orbit of the permutation $I\rightarrow I\dot{+}m$ of $\mathcal D$. It has been proved in \cite{Hyt10} if $(a(I),b(I))=(a(J),b(J))$ for two different cubes $I,J\in\mathcal D_{\mathrm{good}}$, then the cubes satisfy the following $m$-compatibility condition: either the sets $I\cup(I\dot{+}m)$ and $J\cup(J\dot{+}m)$ are disjoint, or one of them, say $I\cup(I\dot{+}m)$, is contained in a dyadic subcube of $J$ or $J\dot{+}m$.

We can hence decompose $\mathcal D_{\mathrm{good}}$ into collections of pairwise $m$-compatible cubes by setting
$$\mathcal D^m_{k,v}:=\{I\in\mathcal D_{\mathrm{good}}:\;a(I)=k,\;b(I)=v\}, \quad k=0,\dotsm,M(m),\quad v=0,1.$$
The total number of these collections is $2(1+M(m))\lesssim (1+\log^+|m|)$.

Note that for fixed $k,v$, $\{h^0_{I\dot{+}m}-h^0_I\}_{I\in\mathcal D^m_{k,v}}$ form a martingale difference sequence. Thus
\begin{align*}
|B^0|&\leq\sum_{\theta\in\{0,1\}^n\setminus\{0\}}\sum_{m\in\mathbb Z^n}\sum_{k,v}|\langle\langle g,\sum_{I\in\mathcal D^m_{k,v}}\langle h^0_{I\dot{+}m},Th^\theta_I\rangle\langle h^\theta_I,f\rangle (h^0_{I\dot{+}m}-h^0_I)\rangle\rangle|\\
&\lesssim \|g\|_{L_2(\mathcal A)}\|f\|_{L_2(\mathcal A)}\sum_{m\in\mathbb Z^n}(1+|m|)^{-n-\alpha}(1+\log^+|m|)\lesssim \|g\|_{L_2(\mathcal A)}\|f\|_{L_2(\mathcal A)},
\end{align*}
where we used again the fact \eqref{decay estimate}

\bigskip

\noindent{\bf Estimate of $C^0_{\mathrm{good}}$}. This term can be dealt with similarly as $B^0_{\mathrm{good}}$ since we can rewrite
\begin{align*}
C^0=\sum_{\theta\in\{0,1\}^n\setminus\{0\}}\sum_{m\in\mathbb Z^n}\sum_{I\in\mathcal D}\langle\langle f^*,\langle h^\theta_{I},Th^0_{I\dot{+}m}\rangle \langle h^\theta_{I},g^*\rangle (h^0_{I\dot{+}m}-h^0_I)\rangle\rangle,
\end{align*}
in the same form with $B^0_{\mathrm{good}}$.

\bigskip

\noindent{\bf Estimate of $P+Q$}. The estimate of $P+Q$ is completed through a similar argument used for Haar multiplier in Proposition \ref{prop:haar1}.

\begin{lemma}\label{lem:haar2}
We have
\begin{align}
|P+Q|\lesssim \|f\|_{L_2(\mathcal A)}\|g\|_{L_2(\mathcal A)}\|T1\|_{\mathrm{BMO}(\mathcal A,\Sigma_\mathcal A)}\lesssim \|f\|_2\|g\|_{2}\|T1\|_{\mathrm{BMO}(\mathbb R^n;\M)}
\end{align}
with some constant independent of $\beta$. Here $\Sigma_\mathcal A$ is the filtration associated to the dyadic system $\mathcal D^\beta$.
\end{lemma}
\begin{proof}
We first rewrite $P,Q$ as follows:
\begin{align*}
P=&\sum_{\theta\in\{0,1\}^n\setminus\{0\}}\sum_{I\in\mathcal D}\langle\langle T^*1,\langle h^\theta_I,f\rangle \langle1_{I}/|I|,g^*\rangle h^\theta_I\rangle\rangle\\
&=\langle T^*1,\sum_{k\in\mathbb Z}\mathbb D_k(f)\mathbb E_{k-1}(g^*)\rangle
\end{align*}
and
\begin{align*}
Q=&\sum_{\theta\in\{0,1\}^n\setminus\{0\}}\sum_{I\in\mathcal D}\langle\langle (T1)^*,\langle 1_{I}/|I|,f\rangle \langle h^\theta_{I},g^*\rangle h^\theta_{I}\rangle\rangle\\
&=\langle (T1)^*,\sum_{k\in\mathbb Z}\mathbb E_{k-1}(f)\mathbb D_{k}(g^*)\rangle.
\end{align*}
Then by the symmetry condition \eqref{symmetry2}, we get
\begin{align*}
|P+Q|\lesssim \|T^*1\|_{\mathrm{BMO}(\mathcal A,\Sigma_\mathcal A)}\|\sum_{k\in\mathbb Z}\mathbb D_k(f)\mathbb E_{k-1}(g^*)+\sum_{k\in\mathbb Z}\mathbb E_{k-1}(f)\mathbb D_{k}(g^*)\|_{\mathrm{H}^m_1(\mathcal A,\Sigma_\mathcal A)},
\end{align*}
which is controlled by
$$\|f\|_{L_2(\mathcal A)}\|g\|_{L_2(\mathcal A)}\|T^*1\|_{\mathrm{BMO}(\mathcal A,\Sigma_\mathcal A)}$$
by the same arguments in the proof of Proposition \ref{prop:haar1}. Noting that for $b=T^*1$
$$\|b\|_{\mathrm{BMO}(\mathbb R^n;\M)}:=\sup_J\big(\frac1{|J|}\int_J\|b-b_J\|^2_\mathcal M\;dx\big)^{\frac12},$$
where the supremum is taken over all the cubes $J$, while in the definition of $\mathrm{BMO}(\mathcal A,\Sigma_\mathcal A)$-norm, $J$ runs over all the elements in $\mathcal D^\beta$.
We get
$$\|T^*1\|_{\mathrm{BMO}(\mathcal A,\Sigma_\mathcal A)}\leq \|T^*1\|_{\mathrm{BMO}(\mathbb R^n;\M)}<\infty$$
by the assumption \eqref{bmo condition3},
and thus finish the proof.
\end{proof}

\begin{remark}
\rm{(i). Let $1<p<\infty$ and $q$ be the conjugate index of $p$. Then using the arguments in the proof of Proposition \ref{prop:haar1}, actually we are able to show
$$|P+Q|\ \ \lesssim\ \left\{ \begin{array}{c} \|f\|_{L_p(\mathcal A)}\|g\|_{\mathrm{H}^c_q(\mathcal A,\Sigma_\mathcal A)}\|T1\|_{\mathrm{BMO}(\mathbb R^n;\M)},\quad\mathrm{whenever}\;2\leq p<\infty; \\ [7pt] \|f\|_{\mathrm{H}^c_p(\mathcal A,\Sigma_\mathcal A)}\|g\|_{L_q(\mathcal A)}\|T1\|_{\mathrm{BMO}(\mathbb R^n;\M)},\quad\mathrm{whenever}\;1<p\leq2. \end{array} \right.$$

(ii). Here it is worthy to point out that in the argument above we did not use directly the boundedness of higher-dimensional Haar multiplier but the proof for one-dimensional Haar multiplier. We will see in the appendix that higher-dimensional Haar multiplier is more difficult to handle. }

\end{remark}

\bigskip

\subsection{Proof of Theorem \ref{thm:general}.} Combine the estimates of the four parts $A_{\mathrm{good}}$, $B^0_{\mathrm{good}}$, $C^0_{\mathrm{good}}$ and $P+Q$, we proved
\begin{align}\label{l2 estimate}
\|Tf\|_{L_2(\mathcal A)}\lesssim\|f\|_{L_2(\mathcal A)}.
\end{align}

To prove other cases $1<p\neq2<\infty$. We will use the atomic characterization of $\mathrm{H}_1^c(\mathbb R^n;\M)$, which was first introduced by one of us \cite{Mei07}. Let us first recall the definition. Let $1\leq p<\infty$. The Hardy space $\mathrm{H}^c_p(\R^n; \M)$ is defined to be the space of functions $f \in L_1(\A)$ for which we have
\begin{align*}
 \|f\|_{\mathrm{H}_p^c(\R^n;\M)}= \Big\| \Big( \int_\Gamma \Big[ \frac{\partial \widehat{f}^*}{\partial t} \frac{\partial \widehat{f}}{\partial t} + \summ_j \frac{\partial \widehat{f}^*}{\partial x_j} \frac{\partial \widehat{f}}{\partial x_j} \Big] (x + \cdot,t) \, \frac{dx dt}{t^{n-1}} \Big)^\frac12 \Big\|_{L_p(\mathcal A)}<\infty,
\end{align*}
with $\Gamma = \{ (x,t) \in \R^{n+1}_+ \, | \ |x| < t \}$ and $\widehat{f}(x,t) = P_t f(x)$ for the Poisson semigroup $(P_t)_{t \ge 0}$.

According to \cite{Mei07}, these Hardy spaces have nice duality and interpolation behavior. Observing that the adjoint operator $T^*$ and its kernel have same properties as $T$ and $K$, thus to finish the proof, it suffices to show
$$T: \mathrm{H}_1^c(\mathbb R^n;\M) \to L_1(\A).$$

On the other hand, $\mathrm{H}_1^c(\mathbb R^n;\M)$ has an atomic characterization. We say that $a \in L_1(\M; L_2^c(\R^n))$ is an atom if there exists a cube $I$ so that
\begin{itemize}
\item $\mathrm{supp} \hskip1pt a \subseteq I$,

\vskip8pt

\item $\displaystyle \int_I a(y) \, dy = 0$,

\item $\|a\|_{L_1(\M;L_2^c(\R^n))} = \displaystyle \tau \Big[ \big( \int_I |a(y)|^2 \, dy \big)^\frac12 \Big] \le \frac{1}{\sqrt{|I|}}$.
\end{itemize}
By \cite[Theorem 2.8]{Mei07}, we have $$\|f\|_{\mathrm{H}_1^c(\R^n;\M)} \, \sim \, \inf \Big\{ \summ_k |\lambda_k| \, \big| \ f = \summ_k \lambda_k a_k \ \mbox{with} \ a_k \ \mbox{atoms} \Big\}.$$

Therefore, we only need to find a uniform upper estimate for the $L_1$ norm of $T(a)$ valid for an arbitrary atom $$\|T(a)\|_{L_1(\mathcal A)} \, \le \, \big\| T(a) 1_{2I} \big\|_{L_1(\mathcal A)}  + \big\| T(a) 1_{\R^n \setminus 2I} \big\|_{L_1(\mathcal A)} .$$ The second term is dominated by
\begin{eqnarray*}
\big\| T(a) 1_{\R^n \setminus 2I} \big\|_{L_1(\mathcal A)}  & = & \tau \int_{\R^n \setminus 2I} \Big| \int_I K(x,y) a(y) \, dy \Big| \, dx \\ & \le & \int_I \Big( \int_{\R^n \setminus 2I} \big\| K(x,y) - K(x,c_I) \big\|_\M \, dx \Big) \tau |a(y)| \, dy \\ [6pt] & \lesssim & \tau \Big( \int_I |a(y)| \, dy \Big) \ \le \ \sqrt{|I|} \tau \Big[ \big( \int_I |a(y)|^2 \, dy \big)^\frac12 \Big] \ \le \ 1,
\end{eqnarray*}
where we have used Kadison-Schwarz inequality in the third inequality. As for the first term, it suffices to show that $T: L_1(\M; L_2^c(\R^n)) \to L_1(\M; L_2^c(\R^n))$, since then we find again
\begin{eqnarray*}
\big\| T(a) 1_{2I} \big\|_{L_1(\mathcal A)}  & = & \tau \Big( \int_{2I} |T(a)(x)| \, dx \Big) \\ & \le & \sqrt{|2I|} \, \tau \Big[ \big( \int_{2I} |T(a)(x)|^2 \, dx \big)^\frac12 \Big] \\ & \lesssim & \sqrt{|2I|} \, \tau \Big[ \big( \int_{I} |a(x)|^2 \, dx \big)^\frac12 \Big] \, \lesssim \, 1.
\end{eqnarray*}
The $L_1(\M; L_2^c(\R^n))$-boundedness of $T$ follows from the duality $$\big\| T(f) \big\|_{L_1(\M;L_2^c(\R^n))} \, \le \, \Big( \sup_{\|g\|_{L_\infty(L_2^c)} \le 1} \big\|T^*(g) \big\|_{L_\infty(\M;L_2^c(\R^n))} \Big) \|f\|_{L_1(\M;L_2^c(\R^n))}.$$ Recall that the adjoint $T^*$ has the same properties as $T$, and thus is bounded on $L_2(\mathcal A)$. This gives rise to
\begin{eqnarray*}
\big\| T^*(g) \big\|_{L_\infty(\M;L_2^c(\R^n))} & = & \Big\| \Big( \int_{\R^n} |T^*(g)(x)|^2 \, dx \Big)^\frac12 \Big\|_\M \\ & = & \sup_{\|u\|_{L_2(\M)} \le 1} \Big( \int_{\R^n} \big\langle |T^*(g)(x)|^2 u,u \big\rangle_{L_2(\M)} \, dx \Big)^\frac12 \\ & = & \sup_{\|u\|_{L_2(\M)} \le 1} \Big( \int_{\R^n} \big\|T^*(gu)(x) \big\|_{L_2(\M)}^2 \, dx \Big)^\frac12 \\ & \lesssim & \sup_{\|u\|_{L_2(\M)} \le 1} \Big( \int_{\R^n} \big\| g(x)u \big\|_{L_2(\M)}^2 \, dx \Big)^\frac12 \\ & = & \Big\| \Big( \int_{\R^n} |g(x)|^2 \, dx \Big)^\frac12 \Big\|_\M.
\end{eqnarray*}
The third identity above uses the right $\M$-module nature of $T$.

\begin{remark}
\emph{It is a quite interesting question to give a direct proof of Theorem \ref{thm:general} in the case $1<p\neq2<\infty$ without using atomic decomposition, interpolation and duality like the one for perfect dyadic CZOs. }
\end{remark}

\section{Appendix}
In this appendix, we show the following commutator estimate.

\begin{theorem}\label{thm:CoEs}
If $b\in\mathrm{BMO}(\mathbb R^n;\mathcal M)$, then the commutator $[R_j,b]$ is bounded on $L_2(\mathcal A)$. Moreover we have the estimate
\begin{align}\label{CoEs}
\|[R_j,b]f\|_{L_2(\mathcal A)}\lesssim \|b\|_{\mathrm{BMO}(\mathbb R^n;\mathcal M)}\|f\|_{L_2(\mathcal A)}.
\end{align}
\end{theorem}

When $n=1$, the Riesz transforms reduce to the Hilbert transform. By noting the boundedness of  the Haar multiplier---Proposition \ref{prop:haar1}, the result has been essentially proven by Petermichl, see Section 2.3 of \cite{Pet00}. When $n>1$, the situation becomes a little bit more complicated. Firstly, reviewing the proof of the boundedness of one-dimensional Haar multiplier, the higher-dimensional case is not trivial since $\mathbb D_kb\mathbb D_kf$ is not $k-1$-th measurable; Secondly, the higher-dimensional Haar systems are also more complicated.

Petermichl-Treil-Volberg in \cite{PTV02}  showed that the Riesz transforms also lie in the closed convex hull of some dyadic shifts. Let us write down explicitly the form of this class of dyadic shifts: Fix a dyadic system $\mathcal D$, let $\theta_0\in \{0,1\}^n$ be the element with first coordinate 1 and others 0,
\begin{align}\label{shift}
Sf=\sum_{I\in \mathcal D}\sum_{\theta\in\{0,1\}^n\setminus\{0\}}\varepsilon^\theta_I\langle  h^{\theta_0}_I,f\rangle h^{\theta}_{\hat{I}},
\end{align}
where $\hat{I}$ is the dyadic father of $I$ and $\varepsilon^\theta_I=\pm1$.

Associated to this fixed dyadic system $\mathcal D$, the Haar multiplier with noncommmuting symbol $b$ is defined as
$$\Lambda_b(f)=\sum_{k\in\mathbb Z}\mathbb D_k(b)\mathbb E_k(f).$$
As in the one-dimensional case---Proposition \ref{prop:haar1}, one also gets

\begin{proposition}\label{prop:haar2}
If $b\in \mathrm {BMO}(\mathbb R^n;\mathcal M)$, then we have
\begin{itemize}
\item $\Lambda_b$ is bounded from $L_p(\mathcal A)$ to $\mathrm{H}^c_p(\A,\Sigma_\A)$ whenever $2\leq p<\infty$;
\item $\Lambda_b$ is bounded from $\mathrm{H}^c_p(\A,\Sigma_\A)$ to $L_p(\mathcal A)$ whenever $1<p\leq2$.
\end{itemize}
\end{proposition}

\begin{proof}
It suffices to show the case $2\leq p<\infty$, since another case can be shown similarly. Let $q$ be the conjugate index of $p$. Let $f\in L_p(\A)$, and $g\in \mathrm{H}^c_q(\A,\Sigma_\A)$. By approximation, we can assume $b$, $f$ and $g$ are ``nice", so that we do not to justify the infinite sum in the following calculations. By duality, it suffices to show
$$|\varphi(\Lambda_b(f)g^*)|\lesssim \|f\|_{L_p(\A)}\|g\|_{\mathrm{H}^c_q(\A,\Sigma_\A)}.$$
Noting that $\Lambda_bf=bf-\sum_k\mathbb E_{k-1}(b)\mathbb D_k(f)$, we have
\begin{align*}
|\varphi(\Lambda_b(f)g^*)|&=|\varphi((bf-\sum_k\mathbb E_{k-1}(b)\mathbb D_k(f))g^*)|\\
&=|\varphi(bfg^*)-\varphi(b\sum_k\mathbb E_{k-1}(\mathbb{D}_{k}(f)\mathbb D_k(g^*)))|\\
&=|\varphi(b(fg^*-\sum_k\mathbb E_{k-1}(\mathbb{D}_{k}(f)\mathbb D_k(g^*))))|.
\end{align*}
Now decompose $fg^*$, one gets
\begin{align*}
&fg^*-\sum_k\mathbb E_{k-1}(\mathbb{D}_{k}(f)\mathbb D_k(g^*))=\sum_k\mathbb D_{k}(\mathbb{D}_{k}(f)\mathbb D_k(g^*))\\
&+(\sum_{k\in\mathbb Z}\mathbb{E}_{k-1}(f)\mathbb{D}_k(g^*)+\sum_{k\in\mathbb Z}\mathbb{D}_{k}(f)\mathbb E_{k-1}(g^*)).
\end{align*}
The second term can be estimated as in the one-dimensional case. For the first term, note that $\mathbb D_{k}(\mathbb{D}_{k}(f)\mathbb D_k(g^*))$ is a martingale difference, by duality and the fact that the dyadic BMO-norm is controlled by usual BMO-norm, we have
\begin{align*}
&|\varphi(b\sum_k\mathbb D_{k}(\mathbb{D}_{k}(f)\mathbb D_k(g^*)))|\lesssim \|b\|_{\mathrm {BMO}(\mathbb R^n;\mathcal M)}\|\sum_k\mathbb D_{k}(\mathbb{D}_{k}(f)\mathbb D_k(g^*))\|_{\mathrm {H}^m_1(\mathcal A, \Sigma_\mathcal A)}\\
&=\|b\|_{\mathrm {BMO}(\mathbb R^n;\mathcal M)}\int_{\mathbb R^n}\sup_{\ell\in\mathbb Z}\|\sum^\ell_{k=-\infty}\mathbb D_{k}(\mathbb{D}_{k}(f)\mathbb D_k(g^*))\|_{L_1(\mathcal M)}dx.
\end{align*}
Splitting
\begin{align*}
\sum^\ell_{k=-\infty}\mathbb D_{k}(\mathbb{D}_{k}(f)\mathbb D_k(g^*))=\sum^\ell_{k=-\infty}\mathbb{D}_{k}(f)\mathbb D_k(g^*)-\sum^\ell_{k=-\infty}\mathbb E_{k-1}(\mathbb{D}_{k}(f)\mathbb D_k(g^*)),
\end{align*}
noting that the first term can be dealt with as in Proposition \ref{prop:haar1}, and we are reduced to show
\begin{align}\label{key3}
\int_{\mathbb R^n}\sup_{\ell\in\mathbb Z}\|\sum^\ell_{k=-\infty}\mathbb E_{k-1}(\mathbb D_k(f)\mathbb {D}_k(g^*))\|_{L_1(\M)}dx\lesssim\|f\|_{L_p(\A)}\|g\|_{\mathrm{H}^c_q(\A,\Sigma_\A)}.
\end{align}
Using twice the H\"older inequalities and vector-valued Doob's inequality \eqref{doob vector}, it suffices to show for any $\ell$
\begin{eqnarray*}
&&\|\sum^\ell_{k=-\infty}\mathbb E_{k-1}(\mathbb D_k(f)\mathbb {D}_k(g^*))\|_{L_1(\M)}\\
&\leq& \|(\sum_{k\in\mathbb Z}|\mathbb D_k(f)|^2)^{1/2}\|_{L_p(\M)}\|(\sum_{k\in\mathbb Z}|\mathbb D_k(g)|^2)^{1/2}\|_{L_q(\M)}.
\end{eqnarray*}
By duality and the H\"older inequality, using the trace-preserving property of conditional expectation,
\begin{align*}
&\|\sum^\ell_{k=-\infty}\mathbb E_{k-1}(\mathbb D_k(f)\mathbb {D}_k(g^*))\|_{L_1(\M)}\\
&=\sup_{u,\;\|u\|_\M\leq1}|\tau(u\sum^\ell_{k=-\infty}\mathbb E_{k-1}(\mathbb D_k(f)\mathbb {D}_k(g^*)))|\\
&=\sup_{u,\;\|u\|_\M\leq1}|\tau(\sum^\ell_{k=-\infty}\mathbb E_{k-1}(u)(\mathbb D_k(f)\mathbb {D}_k(g^*)))|\\
&=\sup_{u,\;\|u\|_\M\leq1}|\tau(\sum^\ell_{k=-\infty}\mathbb {D}_k(g^*)(\mathbb E_{k-1}(u)\mathbb D_k(f)))|\\
&=\sup_{u,\;\|u\|_\M\leq1}|\tau\otimes tr((\sum^\ell_{k=-\infty}\mathbb {D}_k(g^*)\otimes e_{1k})(\sum^\ell_{k=-\infty}\mathbb E_{k-1}(u)\mathbb D_k(f)\otimes e_{k1}))|\\
&\leq \sup_{u,\;\|u\|_\M\leq1}\|\sum^\ell_{k=-\infty}\mathbb {D}_k(g^*)\otimes e_{1k}\|_{L_q(\M\overline{\otimes}\mathcal B(\ell_2))}\|\sum^\ell_{k=-\infty}\mathbb E_{k-1}(u)\mathbb D_k(f)\otimes e_{k1}\|_{L_p(\M\overline{\otimes}\mathcal B(\ell_2))}\\
&\leq\|(\sum_{k\in\mathbb Z}|\mathbb D_k(f)|^2)^{1/2}\|_{L_p(\M)}\|(\sum_{k\in\mathbb Z}|\mathbb D_k(g)|^2)^{1/2}\|_{L_q(\M)}.
\end{align*}
This finishes the proof by noncommutative Burkholder-Gundy inequality.
\end{proof}



Now let us prove Theorem \ref{thm:CoEs}.
\begin{proof}
Since the Riesz transforms \cite{PTV02}  are shown to be in the convex hull of the dyadic shift operators such as \eqref{shift}, it suffices to estimate $[S, b]$ for one fixed dyadic shift operator $S$. Without loss of generality, we can assume $b=b^*$.  Let $f\in L_2(\mathcal A)$. By approximation, we can assume $b$ and $f$ are ``nice" so that we can decompose $bf=\Lambda_bf+R_bf$, where
$$R_bf=\sum_{k}\mathbb E_{k-1}(b)\mathbb D_{k}(f)=\sum_{\theta\in\{0,1\}^n\setminus\{0\}}\sum_{I\in\mathcal D}\langle b\rangle_I\langle h^\theta_I,f\rangle h^\theta_I,$$
with $\langle b\rangle_I=\frac{1}{|I|}\int_I b$. Thus
$$[S,b]f=[S,\Lambda_b]f+[S,R_b]f.$$
Observe that from the $L_2(\mathcal A)$-boundedness of $S$ and $\Lambda_b$ we have
\begin{align*}
\|[S, \Lambda_b]f\|_{L_2(\mathcal A)}\leq 2\|S\|\|\Lambda_b\|\|f\|_{L_2(\mathcal A)}\lesssim \|b\|_{\mathrm {BMO}(\mathbb R^n,\mathcal M)}\|f\|_{L_2(\mathcal A)}.
\end{align*}
For another term, we claim that
\begin{align}\label{key4}
[S,R_b]f=\sum_{\theta,\eta\in\{0,1\}^n\setminus\{0\}}\sum_{I,J\in\mathcal D}\langle h^\eta_J,Sh^\theta_I\rangle(\langle b\rangle_I-\langle b\rangle_J)\langle h^\theta_I,f\rangle h^\eta_J,
\end{align}
from which, we can conclude the proof. Indeed, by the orthogonality of the Haar basis $h^\theta_I$'s,
$$[S,R_b]f=\sum_{\theta',\theta,\eta\in\{0,1\}^n\setminus\{0\}}\sum_{I\in\mathcal D}a^{\theta',\theta,\eta}_{I}(\langle b\rangle_I-\langle b\rangle_{\hat{I}})\langle h^\theta_I,f\rangle h^\eta_{\hat{I}},$$
where
$$a^{\theta',\theta,\eta}_{I}=\varepsilon^{\theta'}_I\langle h^{\theta_0}_I,h^\theta_I\rangle\langle h^\eta_{\hat{I}},h^{\theta'}_{\hat{I}}\rangle,$$
which equals $\pm 1$ or 0.
Then the fact for any $e\in L_2(\mathcal M)$ with norm 1,
$$\|(\langle b\rangle_I-\langle b\rangle_{\hat{I}})e\|^2_{L_2(\mathcal M)}\lesssim \|b\|^2_{\mathrm {BMO}(\mathbb R^n;\mathcal M)}$$
 yields
$$\|[S, R_b]f\|_{L_2(\mathcal A)}\lesssim\|b\|_{\mathrm {BMO}(\mathbb R^n;\mathcal M)}\|f\|_{L_2(\mathcal A)}.$$

Now let us show the formula \eqref{key4}. Note that $[S,R_b]f=SR_bf-R_bSf$. It is straightforward to compute
\begin{align*}
R_b(Sf)&=\sum_{\eta\in\{0,1\}^n\setminus\{0\}}\sum_{J\in\mathcal D}\langle b\rangle_J\langle h^\eta_J,Sf\rangle h^\eta_J\\
&=\sum_{\theta,\eta\in\{0,1\}^n\setminus\{0\}}\sum_{I,J\in\mathcal D}\langle h^\eta_J,Sh^\theta_I\rangle\langle b\rangle_J \langle h^\theta_I,f\rangle h^\eta_J.
\end{align*}
For another term, we test it on $g=\sum_{\eta\in\{0,1\}^n\setminus\{0\}}\sum_{J\in\mathcal D}\langle h^\eta_J,g\rangle h^\eta_J\in L_2(\mathcal A)$, and obtain
\begin{align*}
\langle\langle g, SR_bf\rangle\rangle&=\sum_{\theta,\eta\in\{0,1\}^n\setminus\{0\}}\sum_{I,J\in\mathcal D}\langle\langle \langle h^\eta_J,g\rangle h^\eta_J, \langle b\rangle_I\langle h^\theta_I,f\rangle Sh^\theta_I\rangle\rangle\\
&=\sum_{\theta,\eta\in\{0,1\}^n\setminus\{0\}}\sum_{I,J\in\mathcal D}\langle\langle g,\langle h^\eta_J, Sh^\theta_I\rangle\langle b\rangle_I\langle h^\theta_I,f\rangle h^\eta_J\rangle\rangle,
\end{align*}
 which yields
$$SR_bf=\sum_{\theta,\eta\in\{0,1\}^n\setminus\{0\}}\sum_{I,J\in\mathcal D}\langle h^\eta_J, Sh^\theta_I\rangle\langle b\rangle_I\langle h^\theta_I,f\rangle h^\eta_J.$$
From the above two identities, we get \eqref{key4}.

\end{proof}

\begin{remark}
\emph{(i).
The above argument works also for general dyadic shifts such as those introduced in \cite{LPR10}. But at the time of writing, the authors have no idea how to show similar results for general Calder\'on-Zygmund singular integral operators. }

\emph{(ii). In the framework of noncommutative harmonic analysis, it would be also interesting to show the result for $p\neq2$. But now the proof is not trivial at all. This is related to the last remark of the previous section.}
\end{remark}

\noindent \textbf{Acknowledgement.} Hong is partially supported by the NSF of China-11601396, 11431011. Liu is partially supported by the NSF of China-11501169. Mei is partially supported by NSF DMS-1700171.

\vskip30pt


\begin{thebibliography}{0}

\bibitem {AHMTT02} P. Auscher, S. Hoffman, C. Muscalu, T. Tao and C. Thiele, Carleson measures trees, extrapolation, and $\mathrm{T}(b)$ theorems, Pub. Mat. {46} (2002), 257-325.


 \bibitem{BlPo08} O. Blasco, S. Pott,
Embeddings between operator-valued dyadic BMO spaces,
Illinois J. Math. 52 (3) (2008), 799-814.

\bibitem{BlPo10} O. Blasco, S. Pott,
Operator-valued dyadic BMO spaces, (English summary)
J. Oper. Theory. 63 (2) (2010), 333-347.

\bibitem{Bou86} J. Bourgain, Vector-valued singular integrals and the $H^1-BMO$ duality, In J. A. Chao, W. Woyczy\'nski (eds.), Probability Theory and Harmonic Analysis, Marcel Dekker, New York, 1986, 1-19.

\bibitem{Bur86} D. L. Burkholder, Martingales and Fourier analysis in Banach spaces. In G. Letta, M. Pratelli (eds.), Probability and Analysis, Lecture Notes in Math. 1206, Springer-Verlag, 1986.

\bibitem{Fig90} T. Figiel, Singular integral operators: a martingale approach. In P. F. X. M\"uller, W. Schachermayer (eds.), Geometry of Banach Spaces. Proceedings of the conference held in Strobl, Austria, 1989. London Math. Soc. Lecture Note Ser. 158, Cambridge Univ. Press, 1990.

\bibitem{Cal77} A. P. Calder\'on, Cauchy integrals on Lipschitz curves and related operators, Proc. Nat. Acad. Sci. USA. 74 (1977), 1324-1327.

\bibitem{CXY13} Z. Chen, Q. Xu and Z. Yin, Harmonic analysis on quantum tori, Comm. Math. Phys. 322 (2013), 755-805.

\bibitem{Chr90} M.Christ, Lectures on singular integral operators,CBMS Regional Conference Series in Mathematics, Vol. 77, Amer. Math. Soc., Providence, RI, 1990.

\bibitem{DaJo84} G. David, J.-L. Journ\'e. A boundedness criterion for generalized Calder\'on-Zygmund operators, Ann. Math. 120 (1984), 371-397.


\bibitem{FMS19} T. Ferguson, T. Mei, B. Simanek, $H^\infty$-calculus for semigroup generators on BMO. Adv. Math. 347 (2019), 408-441.

\bibitem{GPTV00} T. A. Gillespie, S. Pott, S. Treil, A. Volberg, The transfer method in estimates of vector Hankel operators, Algebra i Analiz. 12 (6) (2000), 178-193; (English version to appear in St. Petersburg Math. J.)

\bibitem{GJP17} A. Gonzalez-P\'erez, M. Junge, J. Parcet, Singular integrals in quantum Euclidean spaces, Ann. Sci. \'Ec. Norm. Sup\'er. 50 (4) (2017), 879-925.

\bibitem{Har99} A. Harcharras,  Fourier analysis, Schur multipliers on $S_p$ and non-commutative $\Lambda(p)$-sets, Studia Math. 137 (3) (1999), 203-260.

\bibitem{HJP16} G. Hong, M. Junge, J. Parcet, Algebraic Davis decomposition and Asymmetric Doob indqualities, Commun. Math. Phys. 346 (2016), 995-1019.

\bibitem{HJP17} G. Hong, M. Junge, J. Parcet, Asymmetric Doob indqualities in continuous time, J. Funct. Anal. 273 (4) (2017), 1479-1503.

\bibitem{HLMP14} G. Hong, L.D. L\'opez-S\'anchez, J. Martell, J. Parcet, Calder\'on-Zygmund operators associated to
matrix-valued kernels, Int. Math. Res. Not. 5 (2014), 1221-1252.


 \bibitem{HoMa16} G. Hong, T. Ma, Vector valued q-variation for ergodic averages and analytic semigroups, J. Math. Anal. Appl. 437 (2016), 1084-1100.

\bibitem{HoMa17} G. Hong, T. Ma, Vector valued q-variation for differential operators and semigroups I, Math. Z. 286 (2017), 89-127.


\bibitem{HoMe12} G. Hong, T. Mei, John-Nirenberg inequality and atomic decomposition for noncommutative martingales, J. Funct. Anal. 263 (4) (2012), 1064-1097.

\bibitem{HoYi13} G. Hong and Z. Yin, Wavelet approach to operator-valued Hardy spaces, Rev. Mate. Iberoa. 29 (1) (2013), 293-313.

\bibitem{Hyt06} T. Hyt\"onen, An operator-valued $Tb$ theorem, J. Funct. Anal. 234 (2) (2006), 420-463.

\bibitem{Hyt10} T. Hyt\"onen, Vector-valued singular integrals revisited--with random dyadic cubes, Bull. Pol. Acad. Sci. Math. 60 (3) (2010), 269-283.

\bibitem{Hyt12} T. Hyt\"onen, The sharp weighted bound for general Calder\'on-Zygmund operators, Ann. of Math. 175 (3) (2012), 1473-1506.


\bibitem{HW06} T. Hyt\"onen, L. Weis,
A $T1$ theorem for integral transformations with operator-valued kernel,
J. Reine Angew. Math. 599 (2006), 155-200.

\bibitem{Jun02} M. Junge, Doob's inequality for non-commutative martingales, Journal f\"ur die Reine und Angewandte Mathematik.  549 (2002), 149-190.

\bibitem{JLX06} M. Junge, C. Le Merdy, and Q. Xu, $H^\infty$ functional calculus and square functions on non-commutative $L_p$-spaces, Ast\'erisque. 305 (2006).


\bibitem{JuMe10} M. Junge and T. Mei, Noncommutative Riesz transforms--A probabilistic approach, Amer. J. Math. 132 (2010), 611-681.

\bibitem{JuMe12} M. Junge and T. Mei, BMO spaces associated with semigroups of operators, Math. Ann. 352 (3) (2012), 691-743.

\bibitem{JMP14} M. Junge, T. Mei, and J. Parcet, Smooth Fourier multipliers on group von Neumann algebras, Geom. Funct. Anal. 24 (6) (2014), 1913-1980.

\bibitem{JMP15}  M. Junge, T. Mei, and J. Parcet, Noncommutative Riesz transforms--Dimension free bounds and Fourier multipliers, J. Eur. Math. Soc. 20 (3) (2018), 529-595.

\bibitem{JPPP13} M. Junge, C. Palazuelos, J. Parcet and M. Perrin, Hypercontractivity in group von Neumann algebras, Mem. Amer. Math. Soc. 249 (1183) (2017).

\bibitem{JPPPR12} M. Junge, C. Palazuelos, J. Parcet, M. Perrin and E. Ricard, Hypercontractivity for free products,  Ann. Sci. \'Ec. Norm. Sup\'er. 48 (4) (2015), 861-889.

\bibitem{JuPe14} M. Junge, M. Perrin, Theory of $\mathcal H_p$ space for continuous filtrations in von Neumann algebras, Ast\'erisque. 362 (2014).

\bibitem{JuXu03} M. Junge, Q. Xu, Non-commutative Burkholder/Rosenthal
Inequalities, Ann. Prob. 31 (2) (2003), 948-995.

\bibitem{JuXu08} M. Junge and Q. Xu, Noncommutative Burkholder/Rosenthal
inequalities II: applications, Israel J. Math. 167 (2008), 227-282.

\bibitem{LPR10} M. Lacey, S. Petermichl, M. Reguera, Sharp $A_2$ inequality for Haar shift operators, Math. Ann. 348 (2010), 127-141.

\bibitem{MTX06} T. Mart\'{i}nez, J. L. Torrea, Q. Xu, Vector-valued Littlewood-Paley-Stein theory for semigroups, Adv. Math. 203 (2) (2006), 430-475.

\bibitem{MSX19} E. McDonald, F. Sukochev and X. Xiong, Quantum differentiability on quantum tori. Commun. Math. Phys.
(2019).


\bibitem{MOV11}
J. Mateu, J. Orobitg,  J. Verdera,
Estimates for the maximal singular integral in terms of the singular integral: the case of even kernels,
Ann. of Math. 174 (3) (2011), 1429-1483.

\bibitem{Mei06} T. Mei, Notes on matrix valued paraproducts, Indiana Univ. Math. J. 55 (2) (2006), 747-760.

\bibitem{Mei07} T. Mei, Operator valued Hardy spaces, Mem. Amer. Math. Soc. 188 (2007).


\bibitem{Mei072} T. Mei, An extrapolation of operator-valued dyadic paraproducts, J. Lond. Math. Soc. 81 (3) (2007), 650-662.

\bibitem{Mus03} M. Musat, Interpolation Between Non-commutative BMO and
Non-commutative $L_p$-spaces, J. Funct. Analysis. 202 (1) (2003), 195-225.

\bibitem{NPTV02} F. Nazarov, G. Pisier, S. Treil, A. Volberg, Sharp estimates in vector Carleson imbedding theorem and for vector paraproducts, J. Reine Angew. Math. 542 (2002), 147-171.

\bibitem{NTV97} F. Nazarov, S. Treil, A. Volberg, Counterexample to the infinite dimensional Carleson embedding theorem, C.R. Acad. Sci. Paris Ser. I Math. 325 (1997), 383-389.

\bibitem{NTV03} F. Nazarov, S. Treil, A. Volberg, The $Tb$-theorem on non-homogeneous spaces, Acta. Math, 190 (2) (2003), 151-239.

\bibitem{Par09} J. Parcet, Pseudo-localization of singular integrals and noncommutative Calder\'on-Zygmund theory, J. Funct. Anal. 256 (2009), 509-593.

\bibitem{Pet00} S. Petermichl, Dyadic shifts and a logarithmic bound for Hankel operators with matrix symbols, C. R. Acad. Sci. Paris Ser. I Math. 330 (2000), 455-460.

\bibitem{PTV02} S. Petermichl, S. Treil and A. Volberg, Why the Riesz transforms are averages of the dyadic shifts? Proc. of the 6th Inter. Conf. on Harmonic Analysis and Partial Differential Equations (El Escorial, 2000) (2002) 209-228.

\bibitem{Pis98} G. Pisier, Non-commutative Vector Valued $L_p$-Spaces and Completely $p$-Summing Maps, Ast\'{e}risque 247, 1998.

\bibitem{PiXu97} G. Pisier, Q. Xu, Non-commutative martingale inequalities,  Comm. Math. Phys. 189 (1997), 667-698.

\bibitem{PiXu03} G. Pisier, Q. Xu, Non-commutative $L_p$ -spaces, Handbook of the Geometry of Banach Spaces II. In: Johnson, W.B., Lindenstrauss, J.(eds.) pp. 1459-1517. North-Holland (2003).

\bibitem{PS14} S. Pott, A. Stoica,
Linear bounds for Calder\'on-Zygmund operators with even kernel on UMD spaces.
J. Funct. Anal. 266 (5) (2014), 3303-3319.

\bibitem{Ran02}  N. Randrianantoanina, Non-commutative martingale transform, J. Funct. Analysis, 194 (1) (2002), 181-212.

\bibitem{Ran07} N. Randrianantoanina. Conditioned square functions for noncommutative martingales, Ann. Probab. 35 (3) (2007), 1039-1070.

\bibitem{RiXu16}E. Ricard and Q. Xu, A noncommutative martingale convexity inequality, Ann. Probab, 44 (2016), 867-882.

\bibitem{SZ18} F. Sukochev and D. Zanin, Connes integration formula for the noncommutative plane, Commun. Math. Phys.
359 (2018), 449-466.


\bibitem{TrVo97} S. Treil, A. Volberg, Wavelets and the angle between past and future, J. Funct.
Anal. 143 (1997), 269-308.

\bibitem{XXX16} R. Xia, X. Xiong, Q. Xu, Characterizations of operator-valued Hardy spaces and applications
to harmonic analysis on quantum tori, Adv. Math. 291 (2016), 183-227.

\bibitem{XXY17} X. Xiong, Q. Xu, Z. Yin, Sobolev, Besov and Triebel-Lizorkin spaces on quantum tori,  Mem. Amer. Math. Soc. 252 (2018).

\bibitem{Xu98} Q. Xu, Littlewood-Paley theory for functions with values in uniformly convex spaces, J. Reine. Angew. Math. 504 (1998), 195-226.

\end{thebibliography}
\end{document}